\documentclass{article}
\usepackage{graphicx}
\usepackage{amssymb,amsmath}
\usepackage{amsthm}
\usepackage{color,graphicx}
\usepackage{hyperref}
\usepackage{color}
\usepackage{mathabx}
\usepackage{subfigure}
\usepackage{csquotes}
\usepackage{float}
\usepackage{soul}

\usepackage{verbatim}

\usepackage{mathrsfs}
\usepackage[a4paper, left=2.5cm, right=2cm, top=3cm, bottom=3cm]{geometry}

\usepackage{relsize} 

\usepackage{mathtools} 

\newcommand{\be}{\begin{equation}}

\newcommand{\ee}{\end{equation}}

\newcommand{\R}{{\mathbb R}}

\newcommand{\Ker}{{\rm \,Ker}}

\numberwithin{equation}{section}
\numberwithin{figure}{section}
\renewcommand{\theequation}{\thesection.\arabic{equation}}
\newtheorem{theorem}{Theorem}[section]
\newtheorem{proposition}[theorem]{Proposition}
\newtheorem{remark}[theorem]{Remark}
\newtheorem{lemma}[theorem]{Lemma}

\newtheorem{definition}[theorem]{Definition}

\title{Stability of Periodic Waves for the Defocusing Fractional Cubic Nonlinear Schr\"odinger Equation}

\author{Handan Borluk\footnote{\texttt{handan.borluk@ozyegin.edu.tr}} \\  
{\small Department of Basic Sciences, Ozyegin University},
{\small  Cekmekoy, Istanbul,  Turkey} \\
\\
Gulcin M. Muslu\footnote{\texttt{ gulcin@itu.edu.tr}} \\ 
   {\small  Department of Mathematics, Istanbul Technical University,  Maslak, Istanbul,  Turkey.}\\
\\
F\'abio Natali\footnote{\texttt{fmanatali@uem.br}} \\ 
   {\small Department of Mathematics, State University of Maring\'a, Maring\'a, PR, Brazil}\\
}

\begin{document}
\maketitle

\begin{abstract}
In this paper, we determine the spectral instability of periodic odd waves for the defocusing fractional cubic nonlinear Schr\"odinger equation. Our approach is based on periodic perturbations that have the same period as the standing wave solution, and we construct real periodic waves by minimizing a suitable constrained problem. The odd solution generates three negative simple eigenvalues for the associated linearized operator, and we obtain all this spectral information by using tools related to the oscillation theorem for fractional Hill operators. Newton's iteration method is presented to generate the odd periodic standing wave solutions and numerical results have been used to apply the spectral stability theory via Krein signature as established in \cite{KapitulaKevrekidisSandstedeI} and \cite{KapitulaKevrekidisSandstedeII}.  
\end{abstract}

\emph{Keywords:}  defocusing fractional Schr\"odinger equation, periodic solutions via constrained minimization problem, spectral stability, Newton's iteration method. \\

\renewcommand{\theequation}{\arabic{section}.\arabic{equation}}
\setcounter{equation}{0}


\section{Introduction}

\qquad The main goal of this paper is to present new results concerning the existence and spectral stability of periodic standing waves for the defocusing fractional nonlinear Schrödinger equation (dfNLS)
\begin{equation}\label{fNLS1}
	iU_t - (-\Delta)^sU - |U|^2U = 0.
\end{equation}
Here $U=u+iv\equiv(u,v):\mathbb{T}\times \mathbb{R}\longrightarrow\mathbb{C}$ is a complex-valued function and $2\pi$-periodic at the first variable with $\mathbb{T}:= [-\pi, \pi]$. In our context, the fractional Laplacian $(-\Delta)^s$ is defined as a pseudo-differential
operator
\begin{equation}\label{FLaplacian}
\widehat{(-\Delta)^sV}(\xi)=|\xi|^{2s}\widehat{V}(\xi),
\end{equation}
where $\xi \in \mathbb{Z}$ and $s \in (0,1]$ (see \cite{RoncalStinga}). 

\qquad The dfNLS equation \eqref{fNLS1} admits the following conserved quantities $E,F:H^s_{per}\times H_{per}^s \longrightarrow \mathbb{R}$ which are given by
\begin{equation}\label{E}
	E(U)=\frac{1}{2}\int_{-\pi}^{\pi} |(-\Delta)^{\frac{s}{2}} U|^2+\frac{1}{2}|U|^4\; dx,
\end{equation}
and
\begin{equation}\label{F}
	F(U)=\frac{1}{2} \int_{-\pi}^{\pi}|U|^2\; dx.
\end{equation}

\qquad A standing periodic wave solution for the equation \eqref{fNLS1} has the form
\begin{equation}\label{standingwave}
U(x,t)=e^{i\alpha t}\varphi(x),
\end{equation}
where $\varphi: \mathbb{T} \longrightarrow \mathbb{R}$ is a smooth $2\pi$-periodic function and $\alpha \in \mathbb{R}$ represents the wave frequency which is assumed to be negative for now. Substituting \eqref{standingwave} into \eqref{fNLS1}, we obtain the following differential equation with fractional derivative

\begin{equation}\label{EDO1}
(-\Delta)^s\varphi+\alpha \varphi+\varphi^3=0.
\end{equation}

\qquad For $\alpha:=-\omega<0$, we consider the standard Lyapunov functional defined as
\begin{equation}\label{G}
G(U):=E(U)-\omega F(U).
\end{equation}
By \eqref{EDO1}, we obtain $G'(\varphi,0)=0$, that is, $(\varphi,0)$ is a critical point of $G$. In addition, the linearized operator around the pair $(\varphi,0)$ is given by
\begin{equation}\label{matrixop}
 \mathcal{L}:= G''(\varphi,0)=\begin{pmatrix}
\mathcal{L}_1 & 0 \\
0 & \mathcal{L}_2
\end{pmatrix},
\end{equation}
where
\begin{equation}\label{L1L2}
\mathcal{L}_1=(-\Delta)^s-\omega+3\varphi^2
\qquad \text{and} \qquad
\mathcal{L}_2=(-\Delta)^s-\omega+\varphi^2.
\end{equation}
Both operators $\mathcal{L}_1$ and $\mathcal{L}_2$ are self-adjoint  and they are deﬁned in $L^2_{per}$ with dense domain $H^{2s}_{per}$. It is worth mentioning that operator $\mathcal{L}$ in $(\ref{matrixop})$ plays an important role in our study. In order to set our spectral problem concerning periodic waves with respect to perturbation with the same period, we consider the complex evolution $U=(u,v)$ associated with the equation $(\ref{fNLS1})$. To simplify the notation, let us consider $\Phi=(\varphi,0)$ and the perturbation
\begin{equation}\label{U-1}
	U(x,t) = e^{-i \omega t} (\Phi(x) + W(x,t)),
\end{equation} where $W(x,t) =w_1(x,t)+iw_2(x,t)\equiv (w_1(x,t),w_2(x,t))$. Substituting \eqref{U-1} into \eqref{fNLS1} and neglecting all the nonlinear terms, we get the following linearized equation:
\begin{equation}\label{spectral}
	\frac{d}{dt} W(x,t) = J \mathcal{L} W(x,t),
\end{equation}
where $J$ is given by
\begin{equation}\label{J}
	J = \left( \begin{array}{cccc}
		0 & 1  \\
		
		-1 & 0 \end{array}
	\right),
\end{equation} and $\mathcal{L}$ is the diagonal operator given by $(\ref{matrixop})$.\\

\qquad To define the concept of spectral stability within our context, we need to substitute the growing mode solution of the form $W(x,t)=e^{\lambda t}w(x)$ into the linear equation \eqref{spectral} to obtain the following spectral problem
\begin{equation*}
	J \mathcal{L} w = \lambda w.
\end{equation*} 
\indent The definition of spectral stability in our context reads as follows.

\begin{definition}\label{def-spectralstability}
	The periodic wave $\Phi$ is said to be spectrally stable by periodic perturbations that have the same period as the standing wave solution if $\sigma(J \mathcal{L}) \subset i \mathbb{R}$. Otherwise, if there exists at least one eigenvalue $\lambda$ associated with the operator $J \mathcal{L}$ that has a positive real part, $\Phi$ is said to be spectrally unstable.
\end{definition}
\qquad The study of spectral (orbital) stability of periodic standing waves of the form $U(x,t)=e^{i\alpha t}\varphi(x)$ associated to the cubic nonlinear Schr\"odinger equation 
\begin{equation}\label{KF1}
iU_t+U_{xx}+b|U|^2U=0,
\end{equation}
has attracted the interest of a large number of researchers. Let $\alpha>0$ be fixed. The case $b=1$ in $(\ref{KF1})$ represents the \textit{focusing} nonlinearity, while $b=-1$ represents the \textit{defocusing} nonlinearity. In both cases, sufficient conditions and applications of known techniques have been shown to be efficient. For the case $b=1$, the author of \cite{angulo} established the stability properties of periodic standing waves solutions with \textit{dnoidal} profile with respect to perturbations with the same period $L$ by using the ideas introduced in \cite{bona}  and \cite{weinstein}, (see also \cite{gallay1} and \cite{lecoz}). Existence of smooth branches of solutions with \textit{cnoidal} profiles were also reported in \cite{angulo}; however, the author was not able to obtain the  orbital stability/instability in the energy space for these waves. Next, by using the techniques introduced in \cite{grillakis1} and \cite{grillakis2}, the cnoidal waves were shown to be orbitally stable  in \cite{gallay1} and \cite{gallay2} with respect to anti-periodic perturbations, i.e., when $f$ satisfies $f(x+L/2)=-f(x)$ for all $x\in\R$. Spectral stability with respect to bounded or localized perturbations were also reported in \cite{gallay1}. For $\alpha>0$ in a suitable interval $ (0,\alpha^*)$, the authors of \cite{lecoz} have established spectral stability results for the cnoidal waves with respect to perturbations with the same period  $L$ and orbital stability results in the space  constituted by anti-periodic functions with period $L/2$ (see also \cite{natali-moraes-loreno-pastor}). Their proofs rely on proving that the cnoidal waves satisfy a convenient minimization problem with constraints, which yields the orbital stability. The spectral stability follows by relating the coercivity of the linearized action  with the number of eigenvalues with negative Krein signature of $J\mathcal{L}$.

 \qquad The integrability of the equation \eqref{KF1} can be used to determine spectral stability results  of periodic waves. In \cite{deconinckupsal} the authors studied periodic solutions with dnoidal and cnoidal type for the case $b=1$. The same approach was used in \cite{BND} to prove spectral stability results for the case $b=-1$ and snoidal type solutions. The spectral stability presented in both cases was with respect to subharmonic perturbations, that is, perturbation of an integer multiple $n$ times the minimal period of the solution. The authors employed the arguments in \cite{grillakis1} to conclude the orbital stability by considering the orbit generated only by one symmetry of the equation \eqref{KF1}. 

\qquad We present some recent contributions concerning the fractional version 
\begin{equation}\label{KF2}
iU_t-(-\Delta)^sU+b|U|^2U=0,
\end{equation}
of the equation $(\ref{KF1})$. Indeed, when $s \in (0,1)$ and $b=\pm 1$ the orbital stability of real-valued, even and anti-periodic  standing wave solutions $\varphi$ of \eqref{fNLS1} has been studied in \cite{ClaassenJohnson}. The authors determined the existence of real solutions via a minimization problem in the context of anti-periodic functions (denoted by $L^2_a(0,L)$) and they established that the associated linearized operator acting in $L^2_a(0,L)$ is non-degenerate. By the additional assumption $\tfrac{d}{d\alpha} \int_{0}^{L} \varphi^2dx>0$, the authors were able to show that $\varphi$ is orbitally stable with respect to anti-periodic perturbations in a suitable subspace of $H^s(0,L) \cap L^2_a(0,L) $.

\qquad In \cite{natalifracNLS} the authors studied the existence and orbital stability of positive and periodic standing wave solutions of the form $U(x,t)=e^{i\alpha t}\varphi(x)$ for the equation $(\ref{KF2})$ with with $b=1$. The existence of periodic waves was determined by using a minimizing constrained problem in the complex setting and the orbital stability was proved by combining some tools regarding the oscillation theorem for fractional Hill operators and the Vakhitov-Kolokolov condition. The authors also presented a numerical approach to generate the periodic standing wave solutions of $(\ref{KF2})$ with $b=1$ by using Petviashvili's iteration method. It is important to mention that the numerical method has also been used to establish the values of the frequency $\alpha>\frac{1}{2}$ and the index $s>0$ in $(\ref{KF1})$ where the wave $\varphi$ is spectrally (orbitally) stable or not. In fact, if $s\in \left(\frac{1}{4},\frac{1}{2}\right]$ the periodic wave is spectrally (orbitally) unstable. If $s\in [s^*,1],$ the periodic wave is spectrally (orbitally) stable, where $s^*\approx 0.6$. For $s\in\left(\frac{1}{2},s^*\right)$, the authors guaranteed the existence of a critical value $\alpha_c>\frac{1}{2}$ such that the periodic wave is spectrally (orbitally) unstable if $\alpha\in \left(\frac{1}{2},\alpha_c\right)$ and spectrally (orbitally) stable if $\alpha>\alpha_c$.

\qquad Now, we give the main points of our paper. First, we show the existence of an odd periodic two-lobe solution $\varphi$ for the equation \eqref{EDO1}. For this aim, we need to solve the real constrained minimization problem
\begin{equation}\label{minimization1}
\inf\left\{\mathcal E(u)=E(u,0):= \frac{1}{2}\int_{-\pi}^{\pi}((-\Delta)^{\frac{s}{2}} u)^2+\frac{1}{2}u^4\; dx \; ; \; u \in H^s_{per,odd} ,\:  \int_{-\pi}^{\pi} u^2 \; dx=\tau \right\},
\end{equation}
for fixed $\tau>0$, where $s \in \left(\tfrac{1}{4},1\right]$.

\qquad Periodic odd solutions of $(\ref{minimization1})$ are real functions $\varphi$ and therefore the existence of a periodic standing wave having the form $(\ref{standingwave})$ is established without further problems. This fact is different from the approaches in \cite{ClaassenJohnson} and \cite{natalifracNLS} since they obtained, complex periodic solutions of a complex constrained minimization problem. In both cases, they need to assume suitable assumptions in order to get the existence of periodic standing waves of the form $(\ref{standingwave})$ (see \cite[Lemma 2.2]{ClaassenJohnson} and \cite[Remark 3.3]{natalifracNLS}). Our periodic solution obtained from the problem $(\ref{minimization1})$ enables us to consider a real-valued solution $\varphi$ for the problem $(\ref{minimization1})$ which is automatically odd. In addition, we can consider that $\varphi$ has a two-lobe profile for all $\omega>1$ (see Proposition $\ref{theorem1even}$).

\qquad A different way to construct periodic real-valued solutions associated with the equation \eqref{EDO1} is established by using the local and global bifurcation theory as determined in \cite{buffoni-toland}. First, we construct small amplitude periodic solutions in the same way as in \cite{NataliLePelinovsky} (see also \cite{bruell-dhara}) for $\omega>1$ and close to the bifurcation point $1$. Afterwards, we establish sufficient conditions to extend $\omega$ to the whole interval $(1,+\infty)$ by constructing an odd periodic continuous function $\omega \in \left( 1, +\infty \right) \longmapsto \varphi_{\omega} \in H_{per,odd}^{2s}$ where $\varphi_{\omega}$ is a periodic solution of $(\ref{EDO1})$. It is important to mention that the periodic wave obtained by the global bifurcation theory may not have a two-lobe profile, and thus we can choose the periodic waves which arise as a minimum of the problem $(\ref{minimization1})$. The existence of small amplitude waves associated with the Schr\"odinger equation was determined in \cite{gallay1} for the equation $(\ref{KF1})$ with $b=\pm1$. They first show that these waves are orbitally stable within the class of solutions that have the same period. For the case of general bounded perturbations, they prove that the small amplitude travelling waves are stable in the defocusing case and unstable in the focusing case.

\qquad Since the minimizer $\varphi$ of \eqref{minimization1} is a real odd two-lobe solution of $(\ref{EDO1})$, we obtain that $n(\mathcal{L}_1)=1$ (see Lemma $\ref{simpleKerneleven}$), where $n(\mathcal{A})$ denotes the number of negative eigenvalues of a certain linear operator $\mathcal{A}$ (counting multiplicities). In addition, Lemma \ref{simpleKerneleven} also gives that $\ker(\mathcal{L}_1)=[\varphi']$ and we can use the implicit function theorem to obtain, for a fixed value $\omega_0>1$, the existence of an open interval $\mathcal{I}$ containing $\omega_0$ and a smooth function 
\begin{equation}
\mathcal{I}\ni\omega\longmapsto \varphi_{\omega}\in H_{per,odd}^{2s}
\label{smoothcurve123}\end{equation}
that solves equation $(\ref{EDO1})$. Deriving this equation with respect to $\omega\in \mathcal{I}$, it follows that $\mathcal{L}_1(\partial_{\omega}\varphi)=\varphi$, so that $\varphi\in \mbox{range}(\mathcal{L})$. Concerning the linear operator $\mathcal{L}_2$, we obtain by Lemma $\ref{simpleKernel2even}$ that $n(\mathcal{L}_2)=2$ and $\ker(\mathcal{L}_2)=[\varphi]$. Gathering all spectral information regarding $\mathcal{L}_1$ and $\mathcal{L}_2$, we obtain from the fact $\mathcal{L}$ in $(\ref{matrixop})$ has a diagonal form that $n(\mathcal{L})=3$ and $\ker(\mathcal{L})=[(\varphi',0),(0,\varphi)]$.

\qquad The strategy to prove our spectral instability result is based on an adaptation of the arguments in \cite{KapitulaKevrekidisSandstedeI} and \cite{KapitulaKevrekidisSandstedeII}. Let  $z(\mathcal{A})$ denote the dimension of the kernel of a certain linear operator $\mathcal{A}$. Since in our case we have $z(\mathcal{L})=2$, let $\Theta_1=(\varphi',0)$ and $\Theta_2=(0,\varphi)$ represent the elements in $\ker(\mathcal{L})$. Let $V$ be the $2\times 2$ matrix whose entries are given by 
\begin{equation}\label{V}
V_{jl} = ( \mathcal{L}^{-1} J \Theta_j, J \Theta_l )_{L_{per}^2\times L_{per}^n},
\end{equation}
where $1\leq j,l\leq 2$. Thus, $V$ is given by
\begin{equation}\label{V-1}
\begin{array}{lllll}
V &=& \left( \begin{array}{cc}
		(\mathcal{L}^{-1} J \Theta_1, J \Theta_1)_{L_{per}^2} & (\mathcal{L}^{-1} J \Theta_1, J \Theta_2)_{L_{per}^2} \\
		(\mathcal{L}^{-1} J \Theta_2, J \Theta_1)_{L_{per}^2} & (\mathcal{L}^{-1} J \Theta_2, J \Theta_2)_{L_{per}^2}
	\end{array}
\right).\\\\
	&=&\left( \begin{array}{cc}
	( \mathcal{L}_2^{-1} \varphi', \varphi' )_{L^2_{per}} & 0 \\
	0 & ( \mathcal{L}_1^{-1} \varphi,\varphi )_{L^2_{per}}
\end{array}
	\right).
\end{array}
\end{equation}
\qquad On the other hand, the equality
\begin{equation}\label{krein}
	k_r + k_c + k_{-} = n(\mathcal{L}) - n(V),
\end{equation}
is given in \cite{KapitulaKevrekidisSandstedeII} and the left-hand side of the equality $(\ref{krein})$ is called \textit{hamiltonian Krein index}. Concerning operator $\mathcal{L}$ in $(\ref{matrixop})$, let $k_r$ be the number of real-valued and positive eigenvalues (counting multiplicities). The number $k_c$ denotes the number of complex-valued eigenvalues with a positive real part and $k_-$ is the number of pairs of purely imaginary eigenvalues with negative Krein signature of $\mathcal{L}$. Since $k_c$ and $k_-$ are always even numbers, we obtain that if the right-hand side in \eqref{krein} is an odd number, then $k_r \geq 1$ and we have automatically the spectral instability. Moreover, if the difference $n(\mathcal{L}) - n(V)$ is zero, then $k_c = k_- = k_r = 0$ which implies the spectral stability.

\qquad Since $n(\mathcal{L})=3$ and $z(\mathcal{L})=2$, the case $n(V)=3$ cannot be considered according to the square matrix in $(\ref{V-1})$. Now, if $n(V)=0$, we obtain that $n(\mathcal{L})-n(V)=3$ which implies the spectral instability. When $n(V)=1$, we cannot conclude the spectral instability since the difference $n(\mathcal{L})-n(V)=2$ is an even number. The spectral stability result is inconclusive since the values of $k_c$ and $k_{-}$ are always even numbers (we can have zero or two eigenvalues with positive real part associated with the operator $J\mathcal{L}$). However, if $n(V)=2$, then $n(\mathcal{L})-n(V)=1$ and this implies the spectral instability. To obtain a suitable conclusion for the spectral stability, we need to calculate $( \mathcal{L}_2^{-1} \varphi', \varphi' )_{L^2_{per}}$ and $( \mathcal{L}_1^{-1} \varphi,\varphi )_{L^2_{per}}$.
In our approach, we shall consider the restrictions of the linearized operator $\mathcal{L}$ in $(\ref{matrixop})$ to even and odd functions. These operators will be denoted as $\mathcal{L}_{\text{even}}$ and $\mathcal{L}_{\text{odd}}$ and we can conclude in our case that $n(\mathcal{L}_{even})=2$, $n(\mathcal{L}_{odd})=1$, $\ker(\mathcal{L}_{even})=[(\varphi',0)]$, and $\ker(\mathcal{L}_{odd})=[(0,\varphi)]$. Thus, for the  matrix $V$ with these restrictions we have $V_{even}=( \mathcal{L}_2^{-1} \varphi', \varphi' )_{L^2_{per}}$ and $V_{odd}=( \mathcal{L}_1^{-1} \varphi,\varphi )_{L^2_{per}}$, where $V_{even}$ and $V_{odd}$ are, respectively, the restriction of the matrix $V$ to the even and odd periodic functions.

\qquad  To calculate $V_{even}$ and $V_{odd}$ we use a numerical approach.  To the best of our knowledge, the exact solution of the equation \eqref{fNLS1} is known only for $s=1$. Therefore,   we first generate the odd periodic two-lobe solution $\varphi$ by using Newton's iteration method for $s\in (\frac{1}{4},1 ]$. We then evaluate the necessary inner products numerically by using the trapezoidal rule. The numerical approach enables us to conclude that $V_{even}=( \mathcal{L}_2^{-1} \varphi', \varphi' )_{L^2_{per}}$ is positive while $V_{odd}=( \mathcal{L}_1^{-1} \varphi,\varphi )_{L^2_{per}}$ is negative. Since $n(\mathcal{L})=3$ and $n(V)=1$, we see that the difference $n(\mathcal{L})-n(V)=3-1=2$ is even number and the spectral stability is inconclusive. However, if we restrict our analysis to the space $L_{per,odd}^2\times L_{per,odd}^2$ of odd periodic functions, we obtain $n(\mathcal{L}_{odd})=1$ and $V_{odd}>0$. Since the difference  $n(\mathcal{L}_{odd})-n(V_{odd})=1-0=1$ is an odd number, we obtain that the wave $\Phi$ is spectrally unstable. The same stability property occurs if one considers the operator $\mathcal{L}$ in the space $L_{per,even}^2\times L_{per,even}^2$ of even periodic functions. In this case, we have $n(\mathcal{L}_{even})=2$ and $V_{even}<0$. Since the difference  $n(\mathcal{L}_{even})-n(V_{even})=2-1=1$ is also an odd number, we obtain that the wave $\Phi$ is spectrally unstable.

\qquad In both cases,  our main result is given by the following theorem:

\begin{theorem}\label{mainth} Let $s\in\left(\frac{1}{4},1\right]$ and $\omega>1$ be fixed. Consider $\varphi$ as the odd and periodic two-lobe solution for the equation $(\ref{EDO1})$ obtained by the minimization problem $(\ref{minimization1})$. The periodic wave is spectrally unstable.
\end{theorem}

\begin{remark}\label{rem1}
We can employ the abstract approach from \cite{grillakis1} to establish the orbital instability of periodic waves within the energy space $H_{per}^{s}\times H_{per}^{s}$. In fact, the proof of this assertion revolves around demonstrating the orbital instability within the space $H_{per,odd}^{s}\times H_{per,odd}^{s}$ by exclusively taking into account the rotational symmetry. This is because the translational symmetry is not an invariant in this space. Since $n(\mathcal{L}_{odd})=1$ and $\ker(\mathcal{L}_{odd})=[(0,\varphi)]$, we can define a smooth function $\mathsf{d}:(1,+\infty)\longrightarrow\mathbb{R}$ as follows: $\mathsf{d}(\omega)=E(\varphi,0)-\omega F(\varphi,0)=G(\varphi,0)$. Utilyzing equation $(\ref{smoothcurve})$, we deduce from the fact that $(\varphi,0)$ is a critical point of $G$ in $(\ref{G})$ that $\mathsf{d}''(\omega)=-\frac{1}{2}\frac{d}{d\omega}\int_{-\pi}^{\pi}\varphi(x)^2dx=-\frac{1}{2}V_{odd}<0$. By applying the instability theorem from \cite{grillakis1}, we conclude that the periodic wave $\varphi$ is orbitally unstable in both $H_{per,odd}^{s}\times H_{per,odd}^{s}$ and, consequently, in $H_{per}^{s}\times H_{per}^{s}$.
\end{remark}
\qquad Our paper is organized as follows: In Section 2, we give some remarks on the orbital stability and the global well-posedness for the Cauchy problem associated to the equation $(\ref{fNLS1})$. The existence of odd periodic minimizers with a two-lobe profile as well as the existence of small amplitude periodic waves are determined in Section 3. In Section 4, we present spectral properties for the linearized operator related to the dfNLS equation. Finally, our result about orbital instability associated with periodic waves is  shown in Section 5.\\

\textbf{Notation.} For $s\geq0$, the real Sobolev space
$H^s_{per}:=H^s_{per}(\mathbb{T})$
consists of all real-valued periodic distributions $f$ such that
\begin{equation}\label{norm1}
\|f\|^2_{H^s_{per}}:= 2\pi \sum_{k=-\infty}^{\infty}(1+k^2)^s|\hat{f}(k)|^2 <\infty,
\end{equation}
where $\hat{f}$ is the periodic Fourier transform of $f$ and $\mathbb{T}=[-\pi,\pi]$. The space $H^s_{per}$ is a  Hilbert space with the inner product denoted by $(\cdot, \cdot)_{H_{per}^s}$. When $s=0$, the space $H^s_{per}$ is isometrically isomorphic to the space $L^2_{per}:=H^0_{per}$ (see, e.g., \cite{IorioIorio}). The norm and inner product in $L^2_{per}$ will be denoted by $\|\cdot \|_{L_{per}^2}$ and $(\cdot, \cdot)_{L_{per}^2}$, respectively. To avoid overloading of notation, we omit the interval $[-\pi, \pi]$ of the space $H^s_{per}(\mathbb{T})$ and we denote it simply by $H^s_{per}$. In addition, the norm given in \eqref{norm1} can be written as (see \cite{Ambrosio})
\begin{equation}\label{norm}
	\|f\|_{ H^{s}_{per}}^2=\|(-\Delta)^{\tfrac{s}{2}}f\|_{L^2_{per}}^2+\|f\|_{L^2_{per}}^2.
\end{equation}

For $s\geq0$, we denote
$
H^s_{per,odd(even)}:=\{ f \in H^s_{per} \; ; \; f \:\; \text{is an odd(even) function}\},
$
endowed with the norm and inner product in $H^s_{per}$. Since $\mathbb{C}$ can be identified with $\mathbb{R}^2$, notations above can also be used in the complex/vectorial case in the following sense: For $f\in {H}_{per}^s\times H_{per}^s$ we have $f=f_1+if_2\equiv (f_1,f_2)$, where $f_i\in H_{per}^s$, $i=1,2$.


\section{Remarks on the orbital stability and global well-posedness.}
\qquad Our aim in this section is to give a brief remark concerning the orbital stability of periodic waves,  local and global well-posedness for the associated Cauchy problem associated to the dfNLS equation as
\begin{equation}\label{CauchyProblem1}
	\begin{cases}
		iU_t+(-\Delta)^{s} U- |U|^2 U =0,\\
		U(x,0)=U_0(x).
	\end{cases}
\end{equation}
Indeed $U = U(x,t)$ is a solution of (\ref{fNLS1}), so are $e^{-i\zeta}U$ and $U(x-r,t)$ for any  $\zeta, r \in \mathbb{R}$. Considering $U=(u,v)$, we obtain that \eqref{fNLS1} is invariant under the transformations
\begin{equation}\label{T1}
	S_1(\zeta)U := \left(
	\begin{array}{cc}
		\cos{\zeta} & \sin{\zeta} \\
		- \sin{\zeta} & \cos{\zeta}
	\end{array}
	\right) \left(
	\begin{array}{c}
		u \\ v
	\end{array}
	\right)
\end{equation}
and
\begin{equation}\label{T2}
	S_2(r)U := \left(
	\begin{array}{c}
		u(\cdot - r, \cdot) \\
		v(\cdot - r, \cdot)
	\end{array}
	\right).
\end{equation}
The actions $S_1$ and $S_2$ define unitary groups in ${H}^s_{per}\times {H}^s_{per} $ with infinitesimal generators given by \\
$~~~~~~~~~~~~~~~~~~~~~~~~~~
	S_1'(0)U := \left(
	\begin{array}{cc}
		0 & 1 \\
		-1 & 0
	\end{array}
	\right) \left(
	\begin{array}{c}
		u \\
		v
	\end{array}
	\right) = J \left(
	\begin{array}{c}
		u \\
		v
	\end{array}
	\right)
$
and
$
	S_2'(0)U := \partial_x \left(
	\begin{array}{c}
		u \\
		v
	\end{array}
	\right).
$

\qquad Since the equation $(\ref{fNLS1})$ is invariant under the actions of  $S_1$ and $S_2$, we define the orbit generated by
$
\Phi = (\varphi,0)
$
as
\begin{equation*}
	\mathcal{O}_\Phi = \big\{S_1(\zeta)S_2(r)\Phi; \zeta, r \in \R \big\} = \left\{ \left(
	\begin{array}{cc}
		\cos{\zeta} & \sin{\zeta} \\
		-\sin{\zeta} & \cos{\zeta}
	\end{array}
	\right) \left(
	\begin{array}{c}
		\varphi(\cdot - r) \\
		0
	\end{array}
	\right); \; \zeta, r \in \R \right\}.
\end{equation*}

The pseudometric $d$ in ${H}^s_{per}\times {H}^s_{per}$is given by
$
	d(U,W):= \inf \{ \|U - S_1(\zeta)S_2(r)W\|_{{H}^s_{per}\times {H}^s_{per} } ; \; \, \zeta, r \in \R\}.
$
The distance between $U$ and $W$ is the distance between $U$ and the orbit generated by $W$ under the action of rotation and translation, so that
$
d(U,\Phi) = d(U,\mathcal{O}_\Phi).
$

\qquad We now present our notion of orbital stability.

\begin{definition}\label{defstab}
   We say that $\Phi$ is orbitally stable in ${H}^s_{per}\times H_{per}^s$ provided that, given $\varepsilon > 0$, there exists $\delta > 0$ with the following property: if $U_0 \in {H}^s_{per}\times H_{per}^s$ satisfies $\|U_0 - \Phi\|_{{H}^s_{per}} < \delta$, then the global solution $U(t)$ defined in the semi-interval $[0,+\infty)$ satisfies
	$
		d(U(t), \mathcal{O}_\Phi) < \varepsilon$,  for all  $t \geq 0.
	$
	Otherwise, we say that $\Theta$ is orbitally unstable in ${H}^s_{per}\times H_{per}^s$.
\end{definition}

\qquad Now, we present a global well-posedness result in ${H}^s_{per}\times H_{per}^s$.

\begin{proposition} \label{gwp}
Let $s \in \left( \frac{1}{4},1 \right]$ be fixed. The Cauchy problem in $(\ref{CauchyProblem1})$ is globally well-posed in ${H}^{s}_{per}\times {H}^{s}_{per}$. More precisely, for any $U_0 \in {H}^{s}_{per}\times {H}^{s}_{per}$ there exists a unique global  solution $U \in C([0,+\infty), {H}^{s}_{per}\times {H}^{s}_{per})$ such that $U(0)=U_0$ and it satisfies \eqref{fNLS1}. Moreover, for each $T>0$ the mapping
$$
U_0 \in  {H}^{s}_{per}\times {H}^{s}_{per}\longmapsto  U \in C([0,T], {H}^{s}_{per}\times {H}^{s}_{per})
$$
is continuous.
\end{proposition}
\begin{proof}
The existence of local solutions can be established from the arguments in \cite{boling} where the authors have used Galerkin's method and the fact that $E(U)$ is a non-negative conserved quantity.
\end{proof}
\begin{remark}\label{rem12}
We see from Definition $\ref{defstab}$ that one of the requirements to establish the orbital instability in the energy space ${H}^{s}_{per}\times {H}^{s}_{per}$ is the existence of a convenient initial-data $U_0$ and a finite time $T^*>0$ such that $\lim_{t\longrightarrow T^*}||U(t)||_{{H}^{s}_{per}\times {H}^{s}_{per}}=+\infty$. Since $E(U)$ in $(\ref{E})$ is non-negative, we can obtain global solutions in time in ${H}^{s}_{per}\times {H}^{s}_{per}$ by a simple a priori estimate argument and the time $T^*>0$ above does not exist. Thus, the results obtained in Remark $\ref{rem1}$ give us that $\Phi$ is orbitally unstable in ${H}^{s}_{per}\times {H}^{s}_{per}$ even though the evolution $U$ is global in time.
\end{remark}

\section{Existence of periodic waves.}\label{evensolutions}

\qquad This section is devoted to prove the existence of odd periodic waves for the equation \eqref{EDO1} using two approaches. First, we use a variational characterization by minimizing a suitable constrained functional to obtain the existence of odd periodic waves with a two-lobe profile. Second, we present some tools concerning the existence of small amplitude periodic waves using bifurcation theory.

\subsection{Existence of periodic waves via variational approach.}

\qquad In this subsection, we prove the existence of even periodic solutions for \eqref{EDO1}  by considering the variational problem given by \eqref{minimization1}. Before that, we define the concept of solution with \textit{two-lobe} profile.
\begin{definition}\label{single-lobe}
We say that a periodic wave satisfying the equation \eqref{EDO1} has a two-lobe profile if there exists only one maximum and minimum on $[-\pi,\pi]$. Without loss of generality, we assume that the maximum point occurs at $x=\frac{\pi}{2}$ and the minimum point at $x=-\frac{\pi}{2}$.
\end{definition}

\qquad Let $\tau>0$ be fixed. Consider the set
\begin{equation}\label{minimizerset}
\mathcal{Y}_\tau:=\left\{u \in {H}^s_{per,odd} \; ; \; \|u\|_{L_{per}^2}^2=\tau \right\}.
\end{equation}
It is clear that
\begin{equation}\label{positivity}
\mathcal E (u):= E(u,0) \geq 0.
\end{equation}

\qquad One can establish the result of existence as follows:

\begin{proposition}\label{theorem1complex}
Let $s \in \left(\tfrac{1}{4}, 1 \right]$ and $\tau>0$ be fixed. The minimization problem
\begin{equation}\label{minimizationproblem}
\Gamma:= \inf_{u \in \mathcal{Y}_\tau} \mathcal E(u)
\end{equation}
has at least one solution, that is, there exists a real-valued function $\varphi\in\mathcal{Y}_\tau$ such that\linebreak $\mathcal E (\varphi)=\Gamma$. Moreover,  there exists $\omega>0$ such that $\varphi$ satisfies
$$
(-\Delta)^s \varphi -\omega \varphi + \varphi^3 =0.
$$
\end{proposition}
\begin{proof} Using the smoothness of the functional $\mathcal{E}$, we may consider a sequence of minimizers $(u_n)_{n\in \mathbb{N}} \subset Y_\tau$ such that
\begin{equation}\label{10}
\mathcal{E}(u_n) \longrightarrow \Gamma, \; \; \; n \longrightarrow \infty.
\end{equation}

Since $||(-\Delta)^{\frac{s}{2}}u_n||_{L_{per}^2}^2+||u_n||_{L_{per}^2}^2\leq 2\mathcal E(u_n)+\tau$, we obtain by $(\ref{10})$ that the sequence $(u_n)_{n \in \mathbb{N}} \subset \mathbb{R}$ is bounded in ${H}_{per,odd}^s$. For $s\in\left(\frac{1}{4},1\right]$, we see that the Sobolev space ${H}^s_{per,odd}$ is reflexive. Thus, there exists $\varphi \in {H}^s_{per,odd}$ such that (modulus a subsequence),
\begin{equation}\label{11}
u_n \rightharpoonup \varphi \; \text{weakly in} \; {H}^s_{per,odd}.
\end{equation}

Again, for $s \in \left( \tfrac{1}{4},1\right]$  we obtain that the embedding
\begin{eqnarray}\label{12}
{H}^s_{per,odd} \hookrightarrow  {L}^4_{per,odd}\hookrightarrow L_{per,odd}^2
\end{eqnarray}
is compact (see \cite[Theorem 2.8]{BergerSchechter} or \cite[Theorem 5.1]{Amann}). Thus, modulus a subsequence we also have
\begin{equation}\label{13}
u_n \longrightarrow \varphi \; \text{in} \; {L}^4_{per,odd}\hookrightarrow L_{per,odd}^2.
\end{equation}
\qquad Moreover, using the estimate
\begin{eqnarray*}
\bigg|\int_0^L \big(u_n^4-\varphi^4\big)\; dx \bigg| & \leq & \int_0^L \big|u_n^4-\varphi^4\big|\; dx \\
& \leq & \big( ||\varphi^3||_{{L}_{per}^4}+\|\varphi\|_{{L}_{per}^4}^2\,\|u_n\|_{{L}_{per}^4}+\|\varphi\|_{{L}_{per}^4}\,\|u_n\|^2_{{L}_{per}^4}+\|u_n\|_{{L}_{per}^4}^3\big)\|u_n-\varphi\|_{{L}_{per}^4}
\end{eqnarray*}
and \eqref{13}, it follows that $\|\varphi\|_{L_{per}^2}^2=\tau$. 
Furthermore, since $\mathcal{E}$ is  lower semi-continuous, we have
$$
\mathcal{E}(\varphi) \leq \liminf_{n \to \infty} \mathcal{E}(u_n)
$$
that is,
\begin{equation}\label{14}
\mathcal{E}(\varphi) \leq  \Gamma.
\end{equation}

\qquad On the other hand, once $\varphi$ satisfies $\|\varphi\|_{L_{per}^2}^2=\tau$, we obtain
\begin{equation}\label{15}
\mathcal{E}(\varphi) \geq \Gamma.
\end{equation}

Using \eqref{14} and \eqref{15}, we conclude
$$
\mathcal{E}(\varphi)= \Gamma=\inf_{u \in \mathcal{Y}_\tau} \mathcal{E}(u).
$$
In other words, the function $\varphi \in \mathcal{Y}_\tau \subset {H}^s_{per,odd}$ is a minimizer of the problem \eqref{minimizationproblem}. Notice that since $\tau>0$, we see that $\varphi$ is a real-valued function such that $\varphi \notequiv 0$.

\qquad By the Lagrange multiplier theorem, there exists a constant $c_1 \in \mathbb{R}$ such that
$$
(-\Delta)^s \varphi+ \varphi^3=c_1 \varphi.
$$
Denoting $c_1:=\omega$, we see that  $$\int_{-\pi}^{\pi}((-\Delta)^{\frac{s}{2}}\varphi)^2+\varphi^4dx=\omega\int_{-\pi}^{\pi}\varphi^2dx.$$ 

\qquad Thus, $\omega>0$ and $\varphi$ is a periodic minimizer of the problem $(\ref{minimization1})$ satisfying the equation
\begin{equation}\label{EDO2evencomplex}
(-\Delta)^s \varphi -\omega \varphi + \varphi^3 =0.
\end{equation}
\end{proof}


\begin{proposition}[Existence of Odd Solutions]\label{theorem1even}
Let $s \in \left(\tfrac{1}{4}, 1 \right]$ be fixed. Let $\varphi \in H^s_{per,odd}$ be the real-valued periodic minimizer given by the Proposition \ref{theorem1complex}. If $\omega\in (0,1]$, then $\varphi$ is the zero solution of the equation $(\ref{EDO1})$. If $\omega>1$, then $\varphi$ is the odd periodic two-lobe solution for the equation $(\ref{EDO1})$.
\end{proposition}
\begin{proof}
First, by a bootstrapping argument we infer that $\varphi \in H^\infty_{per,odd}$ (see \cite[Propostion 3.1]{Cristofani} and \cite[Proposition 2.4]{NataliLePelinovsky}). Second, the solution can be zero and we need to avoid this case in order to guarantee that the minimizer has a two-lobe profile. Indeed, if $\varphi\equiv0$, the operator  $\mathcal{L}_1$ in $(\ref{L1L2})$ is then given by
\begin{equation}\label{operazero}
\mathcal{L}_1= (-\Delta)^s-\omega.
\end{equation}
\qquad Using the Poincar\'e-Wirtinger inequality, we have that
\begin{equation}\label{poinc}
(\mathcal{L}_1u,u)_{L_{per}^2}=((-\Delta)^{\frac{s}{2}}u,(-\Delta)^{\frac{s}{2}}u)_{L_{per}^2}-\omega||u||_{L_{per}^2}^2\geq (1-\omega)||u||_{L_{per}^2}^2,
\end{equation}
for all $u\in H_{per,odd}^{2s}$. Thus, operator $\mathcal{L}_1$ in $(\ref{operazero})$ is non-negative when $\omega\in (0,1]$, so that $n(\mathcal{L}_{1})=0$ when $\mathcal{L}_1$ is defined over the space $L_{per,odd}^2$. On the other hand, we see that $\varphi$ is also a periodic minimizer of $G$ restricted only to one constraint and it is expected that $n(\mathcal{L}_1)\leq1$. In addition, we will see in Section \ref{spectralanalysis-section} that if $\varphi$ is a nonconstant minimizer then $\mathcal{L}_1\varphi'=0$ which implies, by Sturm's oscillation theorem for fractional linear operators, in fact that ${\rm n}(\mathcal{L}_1)=1$. Thus, we conclude that the zero solution $\varphi\equiv0$ is a minimizer of \eqref{minimizationproblem} only for $\omega \in (0, 1]$ and for $\omega \in \left(1, +\infty \right)$, solution  $\varphi$ is a non-constant minimizer.

\qquad Second, let us consider $\psi:=\varphi\left(\cdot -\frac{\pi}{2}\right)$ a translation of $\varphi$ by a quarter of the period $2\pi$. Since $\varphi$ is odd, we see that $\psi$ is even and it is easy to see that $\mathcal{E} (\psi)=\Gamma$. Using this new minimizer $\psi$, we can consider the even symmetric rearrangements $\psi^{\star}$ associated with $\psi$ and it is well known that such rearrangements are invariant under our constraint $\int_{-\pi}^{\pi}u^2=\tau$ and under the norm in $L_{per}^4$ by using \cite[Appendix A]{ClaassenJohnson}. Moreover, due to the fractional Polya-Szeg\"o Inequality in \cite[Lemma A.1]{ClaassenJohnson} (see also \cite[Theorem 1.1]{Park}), we obtain the following inequality
$$
\int_{-\pi}^\pi \big((-\Delta)^{\frac{s}{2}}\psi^{\star}\big)^2\;dx \leq \int_{-\pi}^\pi\big((-\Delta)^{\frac{s}{2}}\psi\big)^2\;dx .
$$

Thus, by \eqref{minimizationproblem}, we also obtain $\mathcal{E}(\psi^{\star}) = \Gamma$ in the Sobolev $H_{per,even}^s$ with $\psi^{\star}$ being symmetrically decreasing away from the maximum point $x=0$. To simplify the notation, we assume $\psi=\psi^{\star}$, so that $\varphi$ has an odd two-lobe profile according to the Definition $\ref{single-lobe}$.
\end{proof}
\subsection{Small-amplitude periodic waves}

\qquad The existence and convenient formulas for the odd small amplitude periodic waves associated to the equation \eqref{EDO1} will be presented in this section. We show that the local bifurcation theory used to determine the existence of odd small amplitude waves can be extended and the local solutions can be considered as global solutions and they are unique. This fact is very important in our context since it can be used as an alternative form to prove the existence of periodic even solutions (not necessarily having a two-lobe profile) for the equation $(\ref{EDO1})$ when $s\in (0,1]$ and to do so, we use the theory contained in \cite[Chapters 8 and 9]{buffoni-toland}. In addition, the existence of small amplitude periodic waves helps us in the numeric experiments contained in Section 4.

\qquad We will give some steps to prove the existence of small amplitude periodic waves. For $s \in (0,1]$, let $F: H_{per,odd}^{2s} \times (0,+\infty) \longrightarrow L_{per,odd}^2$ be the smooth map defined by
\begin{equation}\label{F-lyapunov}
	F(g,\omega) = (-\Delta)^s g - \omega g + g^3.
\end{equation}
We see that $F(g,\omega) = 0$ if and only if $g \in H_{per,odd}^{2s}$ satisfies \eqref{EDO1} with correspondent frequency of the wave $\omega \in (0, +\infty)$. The Fr\'echet derivative of the function $F$ with respect to the first variable is then given by
\begin{equation}\label{Dg}
	D_g F(g, \omega) f = \left( (-\Delta)^s - \omega + 3g^2 \right) f.
\end{equation}
Let $\omega_0 > 0$ be fixed. At the point $(0, \omega_0)\in H_{per,odd}^{2s} \times (0,+\infty)$, we have that
\begin{equation}
	D_gF(0, \omega_0) = (-\Delta)^s - \omega_0.
\end{equation}

\qquad As far as we can see, the nontrivial kernel of $D_gF(0, \omega_0)$ is determined by odd periodic functions $h \in H_{per}^{2s}$ such that
\begin{equation}
	\widehat{h}(k) (-\omega_0 + |k|^{2s}) = 0,\ \ \ \ \ \ \ k\in\mathbb{Z}.
\end{equation}
It follows that $D_g F(0, \omega_0)$ has the one-dimensional kernel if and only if $\omega_0 = |k|^{2s}$ for some $k \in \mathbb{Z}$. In other words, we have
\begin{equation}
	{\rm Ker}D_gF(0, \omega_0) = [\tilde{\varphi}_k],
\end{equation}
where $\tilde{\varphi}_k(x) = \sin(kx)$.

\qquad We are enabled to apply the local bifurcation theory contained in \cite[Chapter 8.4]{buffoni-toland} to obtain the existence of an open interval $I$ containing $\omega_0 > 0$, an open ball $B(0,r) \subset H_{per,odd}^{2s}$ for some $r>0$ and a unique smooth mapping
$$\omega \in I \longmapsto \varphi:= \varphi_\omega \in B(0,r) \subset H_{per,odd}^{2s}$$
such that $F(\varphi, \omega) = 0$ for all $\omega \in I$ and $\varphi\in B(0,r)$.

\qquad Next, for each $k \in \mathbb{N}$, the point $(0, \tilde{\omega}_k)$ where $\tilde{\omega}_k := |k|^{2s}$ is a bifurcation point. Moreover, there exists $a_0 > 0$ and a local bifurcation curve
\begin{equation}\label{localcurve}
	a \in (0,a_0) \longmapsto (\varphi_{k,a}, \omega_{k,a}) \in H_{per,odd}^{2s} \times (0,+\infty)
\end{equation}
which emanates from the point $(0, \tilde{\omega}_k)$ to obtain odd small amplitude $\frac{2\pi}{k}$-periodic solutions for the equation \eqref{EDO1}. In addition, we have $\omega_{k,0} = \tilde{\omega}_k$, $D_a \varphi_{k,0} = \tilde{\varphi}_k$ and all solutions of $F(g, \omega) = 0$ in a neighbourhood of $(0, \tilde{\omega}_k)$ belongs to the curve in $(\ref{localcurve})$ depending on $a \in (0,a_0)$.



\begin{proposition}
	
	Let $s \in (0,1]$ be fixed. There exists $a_0 > 0$ such that for all $a \in (0,a_0)$ there is a unique even local periodic solution $\varphi$ for the problem \eqref{EDO1} given by the following expansion:
	\begin{equation}\label{varphi-stokes2}
		\varphi(x) = a \sin(x) + \frac{1}{4(3^{2s}-1)}a^3\sin(3x) + \mathcal{O}(a^5) ,
	\end{equation}
	and
	\begin{equation}\label{varphi-stokes3}
		\omega =1 + \frac{3}{4}a^2 + \mathcal{O}(a^4),
	\end{equation}
	For $s\in (\tfrac{1}{4},1]$, the pair $(\varphi, \omega) \in H_{per,odd}^{s} \times (1, +\infty)$ is global in terms of the parameter $\omega > 1$ and it satisfies \eqref{EDO1}.
\end{proposition}

\begin{proof}
The first part of the proposition has been determined in $(\ref{localcurve})$ by considering $k=1$. The expression in \eqref{varphi-stokes2} can be established similarly to \cite[Proposition 3.1]{NataliLePelinovskyKDV}. 
	
	\qquad To obtain that the  local curve \eqref{localcurve} extends to a global one for the case $s\in(\tfrac{1}{4},1]$, we need to prove first that $D_gF(g,\omega)$ given by \eqref{Dg} is a Fredholm operator of index zero. Let us define the set $S = \{(g, \omega) \in D(F) : F(g, \omega) = 0\}$. Consider $(g, \omega) \in H_{per,odd}^{2s} \times (1,+\infty)$ as a solution of $F(g, \omega) = 0$. We have, for $Y:=L_{per,odd}^{2}$ that
	\begin{equation}\label{fredalt}
		\mathcal{L}_{1|_{Y}} \psi \equiv D_g F(g, \omega) \psi = \left( (-\Delta)^s + 3g^2\right) \psi - \omega\psi = 0,\end{equation}
	has two linearly independent solutions and at most one belongs to $H_{per,odd}^{2s}$ (see \cite[Theorem 3.12]{ClaassenJohnson}). If there are no solutions in $H_{per,odd}^{2s}\backslash\{0\}$, then the equation  $\left( (-\Delta)^s - \omega + 3 g^2\right) \psi = f$ has a unique non-trivial solution $\psi \in H_{per,odd}^{2s}$ for all $f \in Y$ since ${\rm Ker}(\mathcal{L}_{1|_Y})^{\bot} = {\rm Range}(\mathcal{L}_{1|_Y}) = Y$.
	
	\qquad On the other hand, if there is a solution $\theta \in H_{per,odd}^{2s}$ we can use the standard Fredholm Alternative to obtain that $(\ref{fredalt})$ has a solution if, and only if,
	$$\int_{-\pi}^\pi \theta(x) f(x) dx = 0,$$
	for all  $f \in Y$. We then conclude in both cases that the Fr\'echet derivative of $F$ in terms of $g$ given by \eqref{Dg} is a Fredholm operator of index zero.
	
	\qquad Let us prove that every bounded and closed subset of $S$ is a compact set on $H_{per,odd}^{2s}\times (1,+\infty)$. For  $g\in H_{per,odd}^{2s}$ and $\omega>1$, we define $\widetilde{F}(g,\omega)=((-\Delta)^s-\omega)^{-1}g^3$. Since $s\in(\tfrac{1}{4},1]$, we see that $\widetilde{F}$ is well defined since $H_{per,odd}^{2s}$ is a Banach algebra, $(g,\omega)\in S$ if and only if $\widetilde{F}(g,\omega)=g$ and $\widetilde{F}$ maps $H_{per,odd}^{2s}\times (1,+\infty)$ into $H_{per,odd}^{4s}$. The compact embedding $H_{per,odd}^{4s}\hookrightarrow H_{per,odd}^{2s}$ shows that $\widetilde{F}$ maps bounded and closed sets in $H_{per,odd}^{2s}\times (1,+\infty)$ into $H_{per,odd}^{2s}$. Thus, if $R\subset S\subset H_{per,odd}^{2s}\times (1,+\infty)$ is a bounded and closed set, we obtain that $\widetilde{F}(R)$ is relatively compact in $H_{per,odd}^{2s}$. Since $R$ is closed, any sequence $\{(\varphi_n,\omega_n)\}_{n\in\mathbb{N}}$ has a convergent sub-sequence in $R$, so that $R$ is compact in $H_{per,odd}^{2s}\times (1,+\infty)$  as desired.
 
	\qquad Finally, the frequency of the wave given by $(\ref{varphi-stokes3})$ is not constant and we are enabled to apply \cite[Theorem 9.1.1]{buffoni-toland} to extend globally the local bifurcation curve given in \eqref{localcurve}. More precisely, there is a continuous mapping
	\begin{equation}\label{globalcurve}
		\left( 1, +\infty \right)\ni\omega \longmapsto \varphi(\cdot,\omega)=\varphi_{\omega} \in H_{per,odd}^{2s}
	\end{equation}
	where $\varphi_{\omega}$ solves equation $(\ref{EDO1})$.
\end{proof}

\begin{remark}\label{remstokes} It is important to mention that $\varphi \in H^{2s}_{per,odd}$ given by \eqref{varphi-stokes2} is a solution of the minimization problem \eqref{minimizationproblem} by using similar arguments as in \cite[Remark 3.2]{NataliLePelinovskyKDV}. \end{remark}

\section{Spectral analysis}\label{spectralanalysis-section}

\qquad Using the variational characterization determined in the last section, we obtain useful spectral properties for the linearized operator $\mathcal{L}$ in $(\ref{matrixop})$ around the periodic wave $\varphi$ obtained in Theorem $\ref{theorem1even}$. Let $s \in \left(\tfrac{1}{4},1\right]$ and $\omega>1$ be fixed. Consider $\varphi \in H^{\infty}_{per,odd}$ as the periodic minimizer obtained by Theorem \ref{theorem1even}. Our intention is to study the spectral properties of the matrix operator
$$
\mathcal{L}=\left(\begin{array}{cccc}\mathcal{L}_1 & 0\\
	0 & \mathcal{L}_2\end{array}\right):H^{2s}_{per} \times H^{2s}_{per} \subset L^2_{per} \times L^2_{per} \longrightarrow L^2_{per} \times L^2_{per},
$$
where $\mathcal{L}_1, \mathcal{L}_2$ are defined by
\begin{equation}\label{L1eL2}
	\mathcal{L}_1=(-\Delta)^s-\omega+3\varphi^2
	\qquad \text{and} \qquad
	\mathcal{L}_2=(-\Delta)^s-\omega+\varphi^2.
\end{equation}

\qquad We see that operators $\mathcal{L}_1$ and $\mathcal{L}_2$ are the real and imaginary parts of the main operator
$\mathcal{L}$.  Using \eqref{positivity}, we obtain the inequality
$$
\mathcal{G}(u):=G(u,0) \leq \mathcal{E} (u).
$$
\begin{lemma}\label{simpleKerneleven}
Let $s \in \left( \tfrac{1}{4},1 \right]$ and $\omega>1$ be fixed. If $\varphi \in H^\infty_{per,odd}$ is the periodic minimizer given in Theorem \ref{theorem1even}, then  $n(\mathcal{L}_1)=1$ and $z(\mathcal{L}_1)=1$. In particular, we have $n(\mathcal{L}_{1,even})=1$, $z(\mathcal{L}_{1,even})=1$, $n(\mathcal{L}_{1,odd})=0$, and $z(\mathcal{L}_{1,odd})=0$
\end{lemma}
\begin{proof}The fact that $\varphi$ is a minimizer of $\mathcal{E}$ defined in $H_{per,odd}^1$, enables us to deduce that $\varphi$ also is a minimizer of $\mathcal{G}$ defined in $H_{per,odd}^1$. By \cite[Theorem 30.2.2]{BlanchardBruning}, we infer
$$
\mathcal{L}_{{1,odd\big|_{ \{\varphi  \}^{\perp}}}} \geq 0,
$$
where $\mathcal{L}_{1,odd}$ is the restriction of $\mathcal{L}_1$ in $L_{per,odd}^2$ (the subspace of $L_{per}^2$ constituted by odd periodic functions). Since $\varphi$ is a minimizer of the constrained variational problem $(\ref{minimizationproblem})$ with only one constraint, we have $n(\mathcal{L}_{1,odd})\leq1$. In addition, $\varphi'$ is even and it has two symmetric zeroes, namely $\pm x_0$, in the interval $[-\pi,\pi]$. By Krein Rutman's theorem, we see that the first eigenvalue of $\mathcal{L}_1$ needs to be associated to an even periodic function and by oscillation theorem, we have $n(\mathcal{L}_1)=n(\mathcal{L}_{1,odd})+n(\mathcal{L}_{1,even})\leq 2$. On the other hand, let us consider $\psi=\varphi(\cdot-\pi/2)$. Function $\psi$ is an even minimizer of the problem 
\begin{equation}\label{minimizationproblem1}
\Gamma:= \inf_{u \in \widetilde{\mathcal{Y}}_\tau} \mathcal E(u),
\end{equation}
where $\widetilde{\mathcal{Y}}_{\tau}=\{u \in {H}^s_{per,even} \; ; \; \|u\|_{L_{per}^2}^2=\tau\}$. Using similar arguments as above, it follows that $\widetilde{\mathcal{L}}_{1,even}=-\partial_x^2-\omega+3\psi^2$ satisfies $n(\widetilde{\mathcal{L}}_{1,even})\leq1$. Again, by Krein Rutman's theorem, it follows that the first eigenvalue of $\widetilde{\mathcal{L}}_1$ needs to be associated with an even periodic function, and thus $n(\widetilde{\mathcal{L}}_{1,even})=1$. Next, by \cite[Lemma 3.3 - (L1)]{HurJohnson}, we have that $\psi'$ corresponds to the lowest eigenvalue of $\widetilde{\mathcal{L}}_{1,odd}$ and it results to be simple. Therefore $n(\widetilde{\mathcal{L}}_{1,odd})=0$, so that $n(\mathcal{L}_1)=1$ and the eigenfunction associated with the negative eigenvalue results to be positive (or negative) and even by an application of the Krein-Rutman theorem.

\qquad We prove that $z(\mathcal{L}_1)=1$. First, again by \cite[Lemma 3.3 - (L1)]{HurJohnson}, we see that $0$ is the first eigenvalue of $\widetilde{\mathcal{L}}_{1,odd}$ and it results to be simple. Using the implicit function theorem and similar arguments as in \cite[Lemma 2.8]{NataliLePelinovskyKDV}, we obtain that for a fixed value $\omega_0>1$ the existence of a open interval $\mathcal{I}$ containing $\omega_0$ and a smooth function 
\begin{equation}
\mathcal{I}\ni\omega\longmapsto\psi(\cdot,\omega):=\psi_{\omega}\in H_{per,even}^{2s}
\label{smoothcurve}\end{equation}
that solves equation $(\ref{EDO1})$. Deriving this equation with respect to $\omega\in \mathcal{I}$, it follows that $\mathcal{L}_1(\partial_{\omega}\psi)=\psi$ and since $\varphi=\psi(\cdot+\pi/2)$, we automatically obtain $\mathcal{L}_1(\partial_{\omega}\varphi)=\varphi$, so that $\varphi\in \mbox{range}(\mathcal{L})$. 

\qquad Suppose the existence of $\omega_0>1$ such that 
$\{\varphi',\bar{y}\}$ is an orthogonal basis for $\ker(\mathcal{L}_1)$. Since $\varphi$ is odd and $\varphi'\in \ker(\mathcal{L}_1)$ is even, we see that $\bar{y}$ is odd and defined in the symmetric interval $[-\pi,\pi]$. The oscillation theorem implies that $\bar{y}$ has exactly two zeroes over the interval $[-\pi,\pi)$ since $\varphi'$ has two symmetric zeroes $\pm x_0$ in the interval $(-\pi,\pi)$. We can suppose, without loss of generality that $\bar{y}<0$ in $(-\pi,0)$ and $\bar{y}>0$ in $(0,\pi)$ and this behaviour is also satisfied by our solution $\varphi$ since we have considered that $\varphi$ has a two-lobe profile  by Proposition $\ref{theorem1even}$,. The fact that $\mathcal{L}_1$ is a self-adjoint operator defined in $L_{per}^2$ with domain $H_{per}^{2s}$, enables us to conclude that $\mbox{range}(\mathcal{L}_1)=[\ker(\mathcal{L}_1)]^{\bot}$, and thus $(\varphi,\bar{y})_{L_{per}^2}=(\mathcal{L}_1(\partial_{\omega}\varphi),\bar{y})_{L_{per}^2}=0$. This leads to a contradiction since we have also $\varphi<0$ in $(-\pi,0)$ and $\varphi>0$ in $(0,\pi)$.
\end{proof}
\begin{lemma}\label{simpleKernel2even}
Let $s \in \left( \tfrac{1}{4},1 \right]$ and $\omega>1$ be fixed. If $\varphi \in H^\infty_{per,odd}$ is the periodic minimizer given by Theorem \ref{theorem1even}, then  $n(\mathcal{L}_2)=2$ and  $z(\mathcal{L}_2)=1$. In particular, we have $n(\mathcal{L}_{2,even})=1$, $z(\mathcal{L}_{2,even})=0$, $n(\mathcal{L}_{2,odd})=1$, and $z(\mathcal{L}_{2,odd})=1$
\end{lemma}
\begin{proof}
	First, we see that $\varphi$ is an odd eigenfunction of $\mathcal{L}_2$ associated to the eigenvalue $0$ having two zeroes in the interval $[0,L)$. From the oscillation theorem for fractional linear operators, we obtain that $0$ needs to be the second or the third eigenvalue of $\mathcal{L}_2$.
 
	\qquad On the other hand,  let $\mathcal{L}_{2,even}$ be the restriction of $\mathcal{L}_2$ in the even sector of $L_{per}^2$. Thus, by Krein-Rutman's theorem we see that the first eigenvalue of $\mathcal{L}_2$ is always simple and it associated to a positive (negative) even eigenfunction, so that $n(\mathcal{L}_{2,even})\geq1$. We obtain by Courant's min-max characterization of the first eigenvalue that
 \begin{equation}\label{minmax1}\lambda_1=\inf \{(\mathcal{L}_{2,odd}u,u)_{L_{per}^2},\ ||u||_{L_{per}^2}=1\}=\inf \{(\mathcal{L}_{1,odd}u,u)_{L_{per}^2}-2(\varphi^2u,u)_{L_{per}^2},\ ||u||_{L_{per}^2}=1\}.\end{equation}
\qquad  Next, $n(\mathcal{L}_1)=1$ and the first negative eigenvalue of $\mathcal{L}_1$ is associated to a even eigenfunction. In addition, we see that $0$ is associated to the eigenfunction $\varphi'$ which is also even and thus, for $u\in H_{per,odd}^{2s}$ such that $||u||_{L_{per}^2}=1$, we obtain
 $(\mathcal{L}_{1,odd}u,u)_{L_{per}^2}>0$ and by $(\ref{minmax1})$, we get
\begin{equation}\label{minmax2}\begin{array}{lllll}\lambda_1&=&\inf \{(\mathcal{L}_{1,odd}u,u)_{L_{per}^2},\ ||u||_{L_{per}^2}=1\}-
2\sup \{(\varphi^2u,u)_{L_{per}^2},\ ||u||_{L_{per}^2}=1\}\\\\
&<&\displaystyle-2\left(\varphi^2\frac{\varphi}{||\varphi||_{L_{per}^2}},\frac{\varphi}{||\varphi||_{L_{per}^2}}\right)_{L_{per}^2}<0.\end{array}\end{equation}
Since $\lambda_1<0$, we obtain by $(\ref{minmax2})$ that $n(\mathcal{L}_{2,odd})\geq1$. The fact  $n(\mathcal{L}_2)=n(\mathcal{L}_{2,odd})+n(\mathcal{L}_{2,even})$ and  the oscillation theorem for fractional linear operators give us that $n(\mathcal{L}_2)=2$ as requested.

	\qquad We prove that $z(\mathcal{L}_2)=1$. Indeed, since $n(\mathcal{L}_{2,odd})=1$, we see that the corresponding eigenfunction $p$ associated to the first eigenvalue of $\mathcal{L}_{2,odd}$ is odd and consequently, $q=p\left(\cdot-\frac{\pi}{2}\right)$ is an even function that changes its sign. Consider $\widetilde{\mathcal{L}}_{2}=(-\Delta)^s-\omega+\psi^2$ a linear operator where $\psi=\varphi\left(\cdot-\frac{\pi}{2}\right)$ is even. By Krein-Rutman's theorem, we have that the first eigenfunction of $\mathcal{L}_2$ is simple and it is associated with a positive (negative) even periodic function and thus, $0$ can not be an eigenvalue associated with $\widetilde{\mathcal{L}}_{2,odd}$. Since $z(\widetilde{\mathcal{L}}_2)=z(\widetilde{\mathcal{L}}_{2,odd})+z(\widetilde{\mathcal{L}}_{2,even})$, we obtain from the fact $\psi$ is even that $z(\widetilde{\mathcal{L}}_2)=z(\widetilde{\mathcal{L}}_{2,even})=1$. Therefore, using the translation transformation $f=g\left(\cdot-\frac{\pi}{2}\right)$, we obtain $z(\mathcal{L}_2)=z(\mathcal{L}_{2,odd})=1$ as requested.
\end{proof}

\qquad As a consequence of Lemma \ref{simpleKerneleven}, we obtain the existence of a smooth curve of positive and periodic solutions $\varphi_{\omega}$ depending on the wave frequency $\omega>1$ all of the with the same period $2\pi$.

\begin{proposition}\label{smoothcurveeven}
Let $s \in \left(\tfrac{1}{4},1\right]$ and $\varphi_0 \in H^{\infty}_{per,odd}$ be the solution obtained in the Proposition \ref{theorem1even} which is associated to the fixed value $\omega_0>1$. Then, there exists a $C^1$ mapping $\omega \in \mathcal{I}  \longmapsto \varphi_\omega \in H^s_{per,odd}$ defined in an open neighbourhood $\mathcal{I} \subset (1,+\infty)$ of $\omega_0>0$ such that $\varphi_{\omega_0}=\varphi_0$.
\end{proposition}
\begin{proof}
The proof follows from the implicit function theorem. The fact has already been used in the proof of Lemma $\ref{simpleKerneleven}$.
\end{proof}

\begin{remark}
We cannot guarantee that for each $\omega \in \mathcal{I}_{\omega_0}$ given by Proposition \ref{smoothcurveeven} that $\varphi_{\omega}$ solves the minimization problem \eqref{minimizationproblem} except at $\omega=\omega_0$.
\end{remark}

\qquad The results determined in this subsection can be summarized in the following proposition:

\begin{proposition}\label{propspec}
	Let $\varphi$ be the two-lobe profile obtained in Proposition $\ref{theorem1even}$. We have that $n(\mathcal{L})=3$ and $\Ker(\mathcal{L})=[(\varphi',0),(0,\varphi)]$.
\end{proposition}
\begin{flushright}
	$\blacksquare$
\end{flushright}

\section{Numerical experiments - Proof of Theorem $\ref{mainth}$}
\qquad In this section, we generate the periodic standing wave solutions of the dfNLS equation by using Newton's iteration method. 
The method is used to construct the standing wave solutions for the focusing fractional NLS equation \cite{klein},  the fractional KdV equation \cite{NataliLePelinovsky} and the fractional modified KdV equation \cite{NataliLePelinovskyKDV}. We then calculate sign of the inner products  $V_{even}=( \mathcal{L}_2^{-1} \varphi', \varphi' )_{L^2_{per}}$ and $V_{odd}=( \mathcal{L}_1^{-1} \varphi,\varphi )_{L^2_{per}}$ for $s\in (\frac{1}{4},1]$, numerically. 

\subsection{Numerical generation of odd periodic waves}
\qquad  Applying the Fourier transform to the equation \eqref{EDO1}, we obtain
\begin{equation} \label{ODE_Fourier}
 F(\widehat{\varphi})=\left(  |\xi|^{2s}-\omega \right) \widehat{\varphi}+\widehat{\varphi^3}=0.
\end{equation}
We choose the space interval as  $[-\pi,\pi] $ and $N=2^{12}$ Fourier modes. Numerically, $\widehat{\varphi}$ is approximated by a discrete Fourier transform. Since $\widehat{\varphi}$ is a vector with size $N\times 1$ we need to solve a nonlinear system \eqref{ODE_Fourier}. Therefore, we employ a Newton iteration method as
\begin{equation}
    \widehat{\varphi}_{n+1}=\widehat{\varphi}_n-\mathcal{J}^{-1}F(\widehat{\varphi}_n).
\end{equation}
Here,  the Jacobian $\mathcal{J}$ is defined by
\begin{equation}
 \mathcal{J} \widehat{Q}  =\left(  |\xi|^{2s}-\omega \right) \widehat{Q} +3\widehat{\varphi^2 Q}
\end{equation}
for  some  vector  $Q$. To avoid the calculation of the inverse of Jacobian directly we use Newton–Krylov method. Therefore, the inverse of the Jacobian is computed by the generalized minimal residual (GMRES) algorithm \cite{saad}.  The iteration is stopped when the residual norm of the numerical solution is of order $10^{-6}$.

\qquad  The periodic standing wave solution of the dfNLS equation with $s=1$  is given in \cite{GallayPelinovsky} as
\begin{equation}\label{snsol}
  \varphi(x)=\eta~\mbox{sn}\left(2 \frac{{\rm K (k)}}{\pi}x,~ k \right),
\end{equation}
where
$ \displaystyle \eta=2\sqrt{2}k\frac{{\rm K (k)}}{\pi}$.
Here  ${\rm K}(k)$ is the complete elliptic integral of first kind and 
$\displaystyle \omega=4(1+k^2)\frac{{\rm K}^2(k)}{\pi^2}$.

\begin{figure}[h]
 \begin{minipage}[t]{0.4\linewidth}
  \includegraphics[width=3.2in]{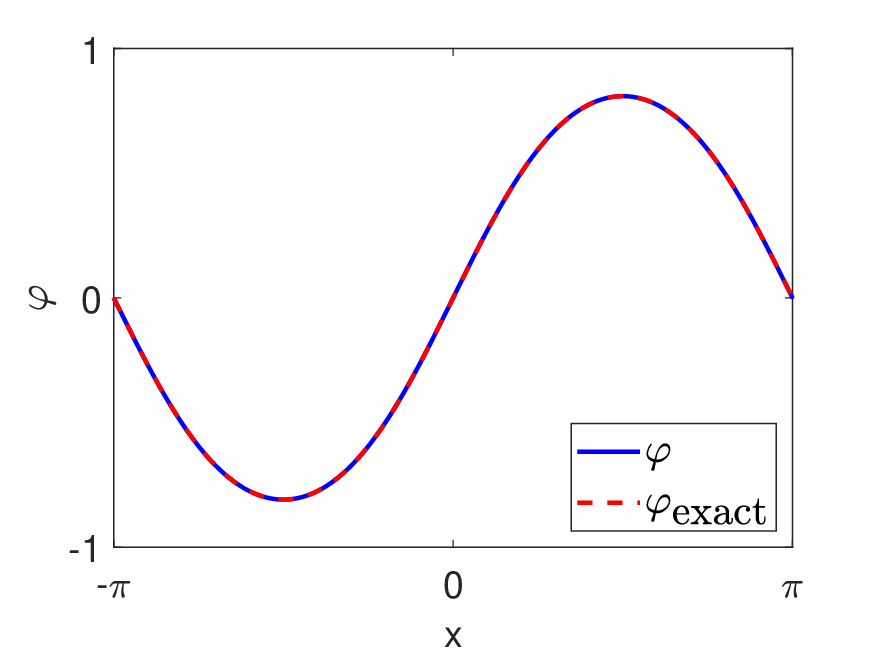}
 \end{minipage}
\hspace{30pt}
\begin{minipage}[t]{0.40\linewidth}
   \includegraphics[width=3.2in]{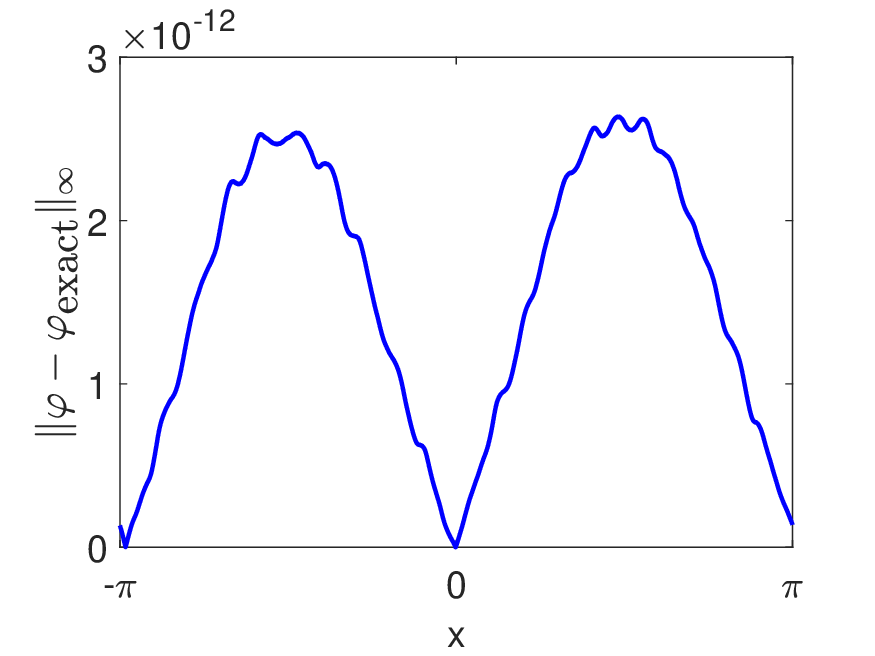}
 \end{minipage}
  \caption{The exact and the numerical solutions of the dfNLS equation  (left)    and the $L_\infty$-error between the exact and numerical solutions (right) where  the wave frequency  $\omega=1.5$ and  $s=1$.}
 \label{exact}
\end{figure}

\qquad In order to test the accuracy of our scheme, we compare the exact solution \eqref{snsol} with the numerical solution obtained by using \eqref{varphi-stokes2} as the initial guess.  In the left panel of   Figure \ref{exact}, we present the exact and numerical solutions for the frequency $\omega=1.5$.  In the right panel, we illustrate  $L_\infty$-error between the exact and numerical solutions.  These results show that our numerical scheme captures the solution remarkably well.

\qquad The exact solutions of the dfNLS equation are not known for $s\in (0, 1)$. The left panel of  Figure \ref{wave_profiles} shows the numerically generated periodic wave profiles for several values of $s$ with $\omega=1.5$. We do not observe any significant change in the wave profiles for different values of $s$. In the right panel of  Figure \ref{wave_profiles} we present numerical wave profiles for various values of $\omega$, with fixed $s=0.5$. The results indicate that the amplitude of the wave increases with the increasing values of $\omega$.

\begin{figure}[h]
 \begin{minipage}[t]{0.4\linewidth}
  \includegraphics[width=3.2in]{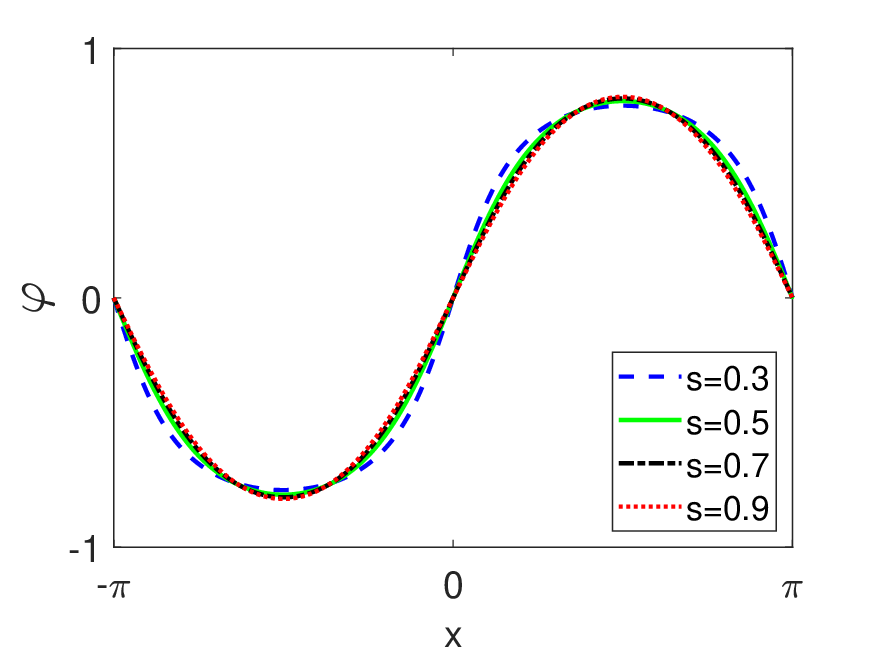}
 \end{minipage}
\hspace{30pt}
\begin{minipage}[t]{0.40\linewidth}
   \includegraphics[width=3.2in]{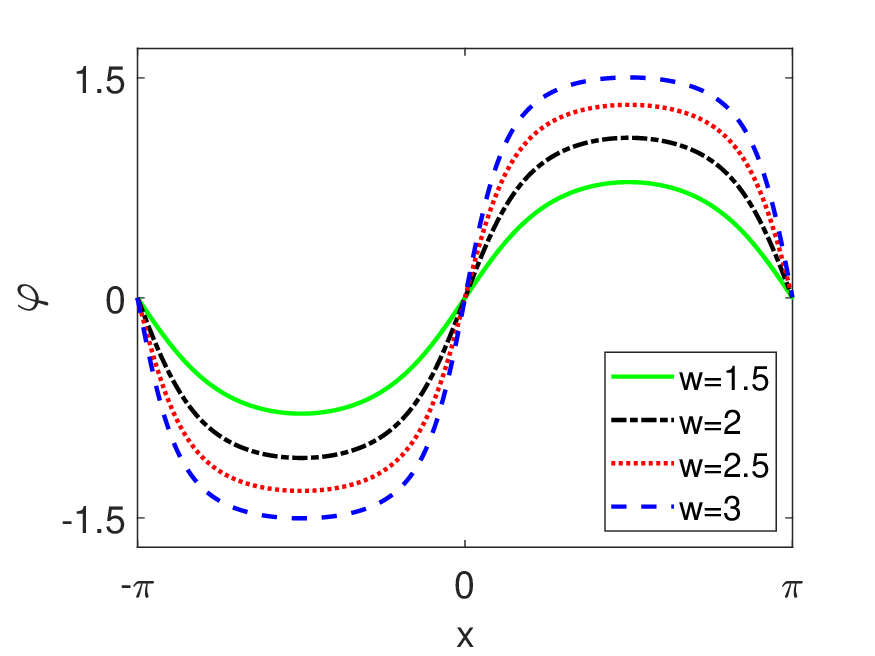}
 \end{minipage}
  \caption{Numerical wave profiles for various values of $s$, with fixed wave frequency $\omega=1.5$ (left) and numerical wave profiles for various values of $\omega$ with fixed $s=0.5$ (right).}
 \label{wave_profiles}
\end{figure}

\subsection{Numerical results for stability}

\qquad In this section, we numerically determine the behavior of the inner products $( \mathcal{L}_1^{-1} \varphi,\varphi )_{L^2_{per}}$ and $( \mathcal{L}_2^{-1} \varphi', \varphi' )_{L^2_{per}}$ in order to conclude the spectral instability. In fact, since $z(\mathcal{L}_1)=1$ and $\ker(\mathcal{L}_1)=[\varphi']$, we obtain that $\varphi\in \rm{range}(\mathcal{L}_1)=\ker(\mathcal{L}_1)^{\bot}$. Therefore, there exist a unique $\chi\in H_{per,odd}^{2s}$ such that $\mathcal{L}_1\chi=\varphi$.  On the other hand, by deriving equation $(\ref{EDO1})$ with respect to $\omega$ we have
\begin{equation*}
(~(-\Delta)^s- \omega  +3\varphi^2) \frac{d\varphi}{dw}=\varphi 
\end{equation*}
which yields $\displaystyle \frac{d\varphi}{dw}=\mathcal{L}_1^{-1} \varphi=\chi$ by uniqueness. Taking the inner product with $\varphi$ gives
\begin{equation*}
    ( \mathcal{L}_1^{-1} \varphi,\varphi )_{L^2_{per}}= 
    \frac{1}{2} \frac{d}{dw} \| \varphi\|^2_{L^2_{per}}.
\end{equation*}

\begin{figure}[h]
 \begin{minipage}[t]{0.4\linewidth}
  \includegraphics[width=3.2in]{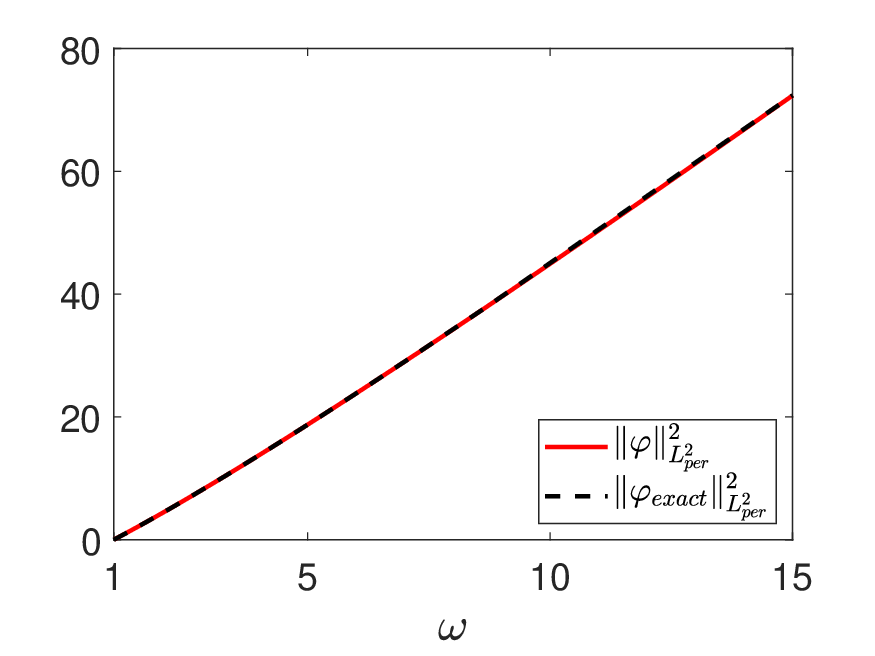}
 \end{minipage}
\hspace{30pt}
\begin{minipage}[t]{0.40\linewidth}
   \includegraphics[width=3.2in]{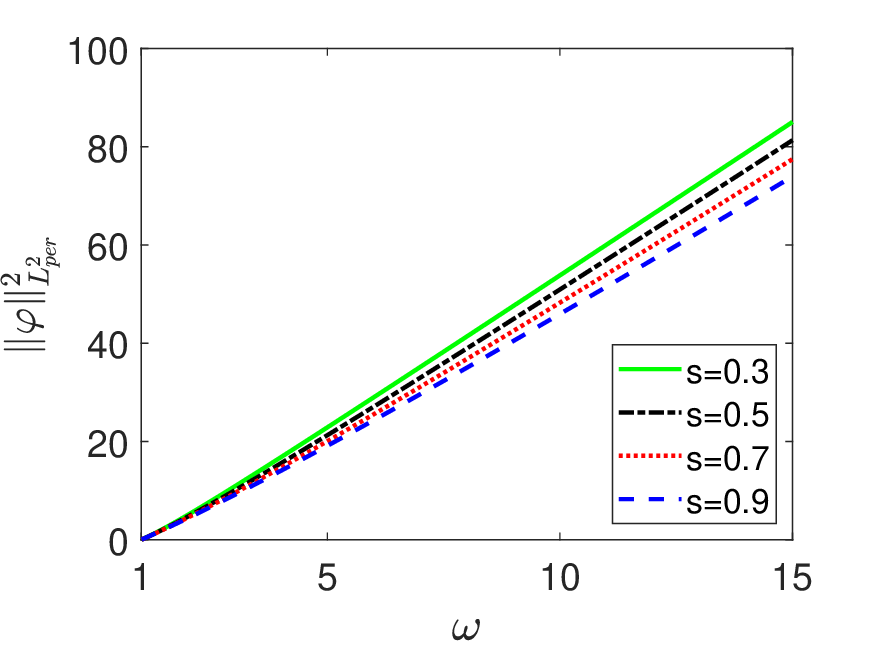}
 \end{minipage}
  \caption{The variation of $\| \varphi \|_{L_{per}^2}^2$ by  $\omega$ for  $s=1$ evaluated by using exact and numerical solutions (left) and the variation of $\| \varphi \|_{L_{per}^2}^2$ by  $\omega$ for several values of $s$ (right).}
 \label{V-odd-numeric}
\end{figure}

\begin{figure}[ht]
 \begin{minipage}[t]{0.45\linewidth}
   \includegraphics[width=3.1in]{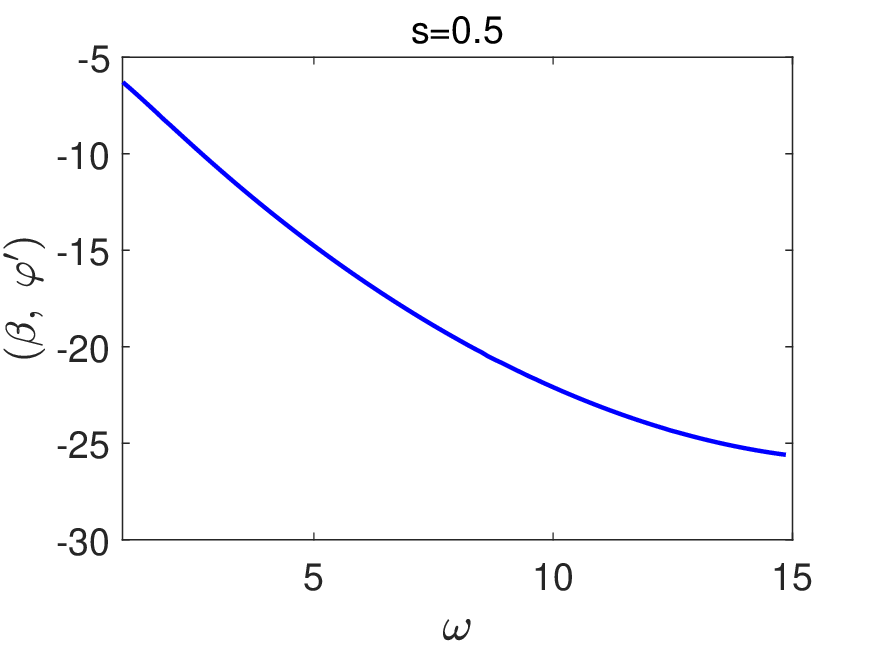}
 \end{minipage}
\hspace{30pt}
\begin{minipage}[t]{0.45\linewidth}
   \includegraphics[width=3.1in]{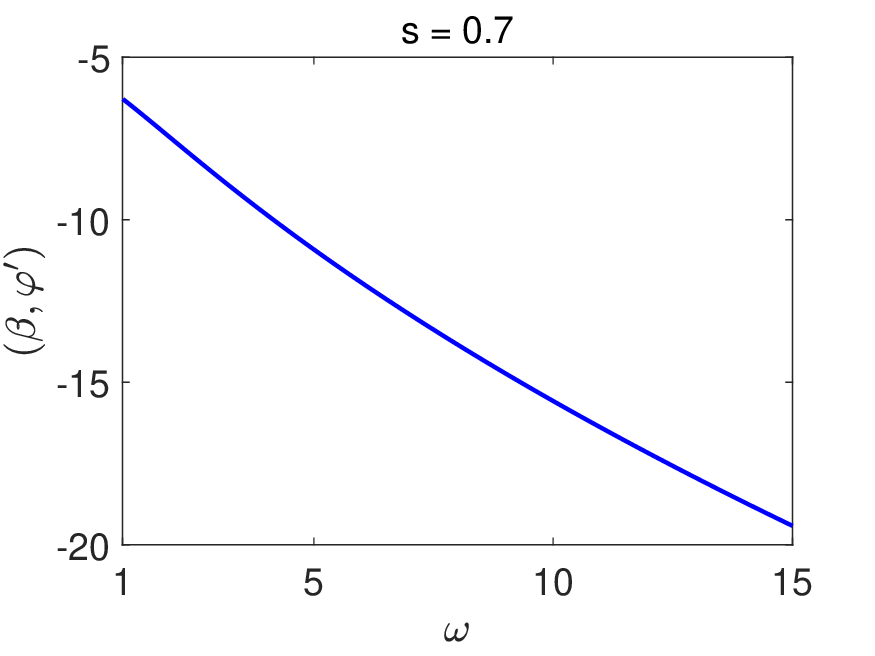}
 \end{minipage}
 \hspace{30pt}
\begin{minipage}[t]{0.45\linewidth}
   \includegraphics[width=3.1in]{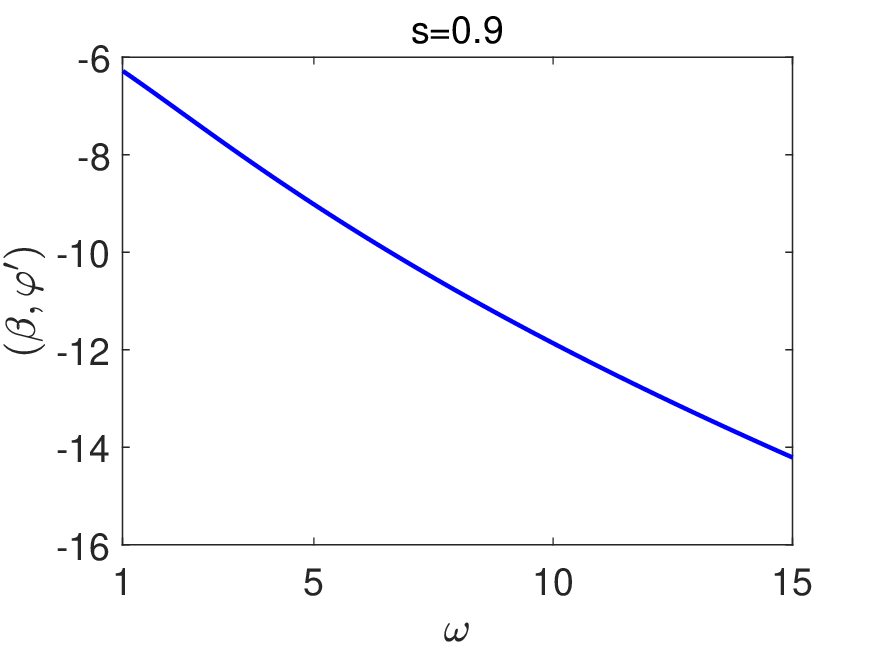}
 \end{minipage}
 \hspace{40pt}
\begin{minipage}[t]{0.45\linewidth}
   \includegraphics[width=3.1in]{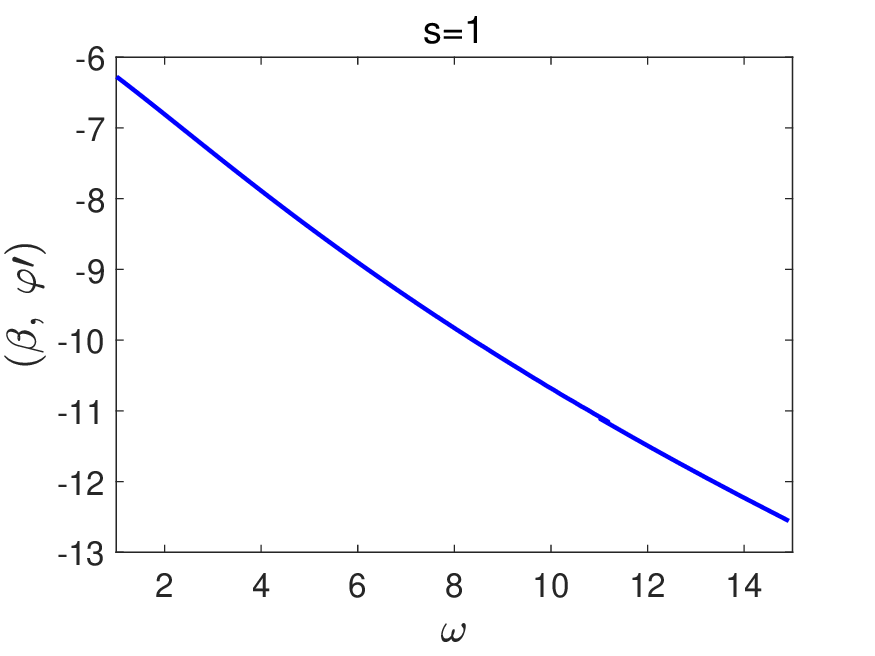}
 \end{minipage}
 \caption{ The variation of $( \mathcal{L}_2^{-1} \varphi', \varphi' )_{L^2_{per}}$ with $\omega$ for $s=0.5,\ 0.7,\ 0.9$, and 1.}
 \label{V-even-numeric}
\end{figure}

\qquad To determine  the behavior of the inner products $( \mathcal{L}_1^{-1} \varphi,\varphi )_{L^2_{per}}$ and $( \mathcal{L}_2^{-1} \varphi', \varphi' )_{L^2_{per}}$ we first obtain  the wave profile $\varphi$  numerically when $\omega>1$.
For small values of $\omega$, we use \eqref{varphi-stokes2} as the starting iteration.   As it is seen from Figure \ref{wave_profiles}, the amplitude of the periodic wave is increasing for increasing values of $\omega$. Therefore, the small amplitude solution \eqref{varphi-stokes2} can not be used as an initial iteration for larger values of $\omega$. For this reason, we use a continuation method, i.e., we use the numerical solution for the previous $\omega$ as an initial iteration and then the solutions are uniquely continued in $\omega$. Next, we use the trapezoidal rule to approximate the integral  $\| \varphi\|^2_{L^2_{per}} $ for each $\omega>1$.  The value of the inner products obtained by using the numerical wave profile and the exact solution \eqref{snsol} are compared in the left panel of Figure \ref{V-odd-numeric}.  We observe that the results coincide very well. 
The right panel of the figure shows the inner product $\| \varphi\|^2_{L^2_{per}} $ for several values of $s$. The numerical results indicate that   $\| \varphi\|^2_{L^2_{per}} $ is  an increasing function of $\omega>1$ therefore the sign of the inner product 
$( \mathcal{L}_1^{-1} \varphi,\varphi )_{L^2_{per}}$ is positive of $s\in (\frac{1}{4},1 ]$. By Lemmas $\ref{simpleKerneleven}$ and $\ref{simpleKernel2even}$, we see that $n(\mathcal{L}_{odd})=1$ and since $V_{odd}=( \mathcal{L}_1^{-1} \varphi,\varphi )_{L^2_{per}}$ is positive, the difference $n(\mathcal{L}_{odd})-n(V_{odd})=1-0=1$ is an odd number. Thus, the wave $\Phi$ is spectrally unstable.

\qquad Now, we compute the sign of $( \mathcal{L}_2^{-1} \varphi', \varphi' )_{L^2_{per}}$ for different values of $s$.  In fact, since $z(\mathcal{L}_{2})=1$ and $\ker(\mathcal{L}_2)=[\varphi]$, we obtain that $\varphi'\in \rm{range}(\mathcal{L}_2)=\ker(\mathcal{L}_2)^{\bot}$. Therefore, there exist a unique $\beta\in H_{per,even}^{2s}$ such that $\mathcal{L}_2\beta=\varphi'$.  Hence, we need to solve  
\begin{equation} \label{eq-beta}
   (-\Delta)^s\beta-\omega\beta+\varphi^2\beta=\varphi'.
\end{equation}
\qquad Applying the Fourier transform to \eqref{eq-beta} we obtain,
\begin{equation} \label{ODE_Fourier_beta}
\left(  |\xi|^{2s}-\omega \right) \widehat{\beta}+\widehat{\varphi^2 \beta}-i\xi \widehat{\varphi}=0.
\end{equation}
To solve \eqref{ODE_Fourier_beta} we use a Newton iteration method as described above. Figure \ref{V-even-numeric} presents the sign of $( \mathcal{L}_2^{-1} \varphi', \varphi' )_{L^2_{per}}$ for  several values of $s$. Numerical results show that the inner product is negative for all $s\in (\frac{1}{4},1]$. By Lemmas $\ref{simpleKerneleven}$ and $\ref{simpleKernel2even}$, we see that $n(\mathcal{L}_{even})=2$ and since $V_{odd}=( \mathcal{L}_1^{-1} \varphi,\varphi )_{L^2_{per}}$ is negative, the difference $n(\mathcal{L}_{even})-n(V_{even})=2-1=1$ is an odd number. Therefore, the wave $\Phi$ is spectrally unstable.
\newpage

\section*{Acknowledgments}
F. Natali is partially supported by CNPq/Brazil (grant 303907/2021-5).

\end{document}